\documentclass{svjour3}

\smartqed

\journalname{}

\usepackage{amssymb,amsmath,algorithm}

\usepackage{float}

\usepackage{graphicx}
\DeclareGraphicsExtensions{.eps}
\usepackage{epsfig}
\usepackage{epstopdf}
\usepackage{amsmath}
\usepackage{amssymb}
\usepackage{subfig}
\usepackage{rotating}
\usepackage{wrapfig}
\usepackage{makeidx,epsfig}
\usepackage[margin=1.0in]{geometry}
\usepackage{amssymb,amsfonts}
\usepackage{amsmath}
\usepackage{amsbsy}
\usepackage{verbatim}
\usepackage[bottom]{footmisc}
\usepackage{algorithm}
\usepackage{algorithmicx}
\usepackage[noend]{algpseudocode}
\algdef{SE}{Begin}{End}{\textbf{begin}}{\textbf{end}}
\usepackage{colortbl,xcolor,pgfplots}
\usepackage{algorithm,algorithmicx,algpseudocode}
\usepackage{xcolor}
\usepackage{makeidx}
\usepackage{capt-of}
\usepackage[T1]{fontenc}
\usepackage{amsfonts}
\usepackage{amscd}
\usepackage{amssymb}
\usepackage{tabularx}
\usepackage{amssymb,bezier,graphicx,}
\usepackage{times,amssymb}
\usepackage{graphicx}
\usepackage{color}
\usepackage{cancel}
\usepackage{fancybox}
\usepackage{fancyhdr}
\usepackage{hyperref}
\usepackage{tikz}
\usepackage{enumerate}
\usepackage{epstopdf}
\usepackage{array}
\usepackage{slashbox}
\usepackage{verbatim}
\usepackage{float}
\usepackage[colorinlistoftodos]{todonotes}
\usepackage{framed}
\usepackage{anysize}
\usepackage{pdfpages}
\usepackage{stackrel}
\usepackage{booktabs}
\usepackage[mathscr]{eucal}
\usepackage{bm}
\usepackage{multirow}
\usepackage{caption}
\usepackage{mathrsfs} 
\usepackage{xcolor}
\usepackage[normalem]{ulem}
\usepackage{marginnote}

\newtheorem{prop}[theorem]{Proposition}
\newtheorem{cor}[theorem]{Corollary}

\newcommand\MR[1]{{\color{black} #1}}
\newcommand\Mr[1]{{\color{black} #1}}
\newcommand\MRM[1]{{\color{black} #1}}
\newcommand\mrm[1]{{\color{black} #1}}
\newcommand\doubts[1]{{\color{black} #1}}
\newcommand\JP[1]{{\color{black} #1}}
\newcommand\mr[1]{{\color{black} #1}}
\newcommand\mrd[1]{{\color{black} #1}}

\begin{document}
\title{A combinatorial optimization approach to scenario filtering in portfolio selection}

\author{Justo Puerto, Mois\'{e}s Rodr\'{i}guez-Madrena \\ Federica Ricca, Andrea Scozzari}
\institute{$^1$ Institute of Mathematics University of Seville (IMUS), Seville, Spain, puerto@us.es\\
$^2$ Department of Statistics and OR, Universidad de Sevilla, Seville, Spain, madrena@us.es\\
$^3$ MEMOTEF, Faculty of Economics, Universit\`{a} degli Studi di Roma La Sapienza, Italy, federica.ricca@uniroma1.it\\
$^4$ Faculty of Economics, Universit\`{a} degli Studi Niccol\`{o} Cusano Roma, Italy, andrea.scozzari@unicusano.it}

\date{\today}
\maketitle

\begin{abstract}
Recent studies stressed the fact that covariance matrices computed from empirical financial time series appear to contain a high amount of noise. This makes the classical Markowitz Mean-Variance Optimization model unable to correctly evaluate the performance associated to selected portfolios. Since the Markowitz model is still one of the most used practitioner-oriented tool, several filtering methods have been proposed in the literature to fix the problem. Among them, the two most promising ones refer to the Random Matrix Theory or to the Power Mapping strategy. The basic idea of these methods is to transform the correlation matrix maintaining the Mean-Variance Optimization model. However, experimental analysis shows that these two strategies are not adequately effective when applied to real financial datasets.

In this paper we propose an alternative filtering method based on Combinatorial Optimization. We advance a new Mixed Integer Quadratic Programming model to filter those observations that may influence the performance of a portfolio in the future. We discuss the properties of this new model and we test it on some real financial datasets. We compare the out-of-sample performance of our portfolios with the one of the portfolios provided by the two above mentioned alternative strategies. We show that our method outperforms them. Although our model can be solved efficiently with standard optimization solvers the computational burden increases for large datasets. To overcome this issue we also propose a heuristic procedure that empirically showed to be both efficient and effective.

\end{abstract}

\section{Introduction and motivation}\label{Intro}

The Mean-Variance Optimization (MVO) approach introduced by Markowitz in 1952 \cite{Markowitz1952} has dominated the asset allocation process for more than 50 years. The Markowitz model minimizes the variance of the return of a portfolio under the request of a fixed expected return level. In spite of its success, MVO has received many criticisms, in particular, related to the fact that returns are not Normally distributed, and the corresponding sampled correlation matrix is biased. Empirical studies have established that the distribution of speculative assets' returns tends to have \emph{fatter tails} than the Gaussian distribution (see, e.g., \cite{Fama1965,JansenVries_1991,Mandelbrot1963}). Fat tail distributions are related to high \emph{kurtosis} distributions for which extreme events, characterized by small probabilities in theory, empirically seem to occur more often than what the Normal distribution predicts. After the recent financial crisis of 2008, many investors moved their attention to how to deal with the risk associated to these extreme events.

In a recent paper by Stoyanov et al. (see \cite{Stoyanov_etal_2011} and the references therein), the authors discuss and compare some popular methods for fat tails modeling based on full distribution modeling and extreme value theory. They conclude that the best approach should be to extend the Gaussian model incorporating methods for handling fat tails, and then testing their performance on real financial datasets. In any case, there is no evidence on which is the best among all the families of fat tailed models, as there may be different families of fat tailed distributions which are statistically equivalent.
\vskip 8 pt
\noindent An alternative analysis on the effects of noise in the estimated correlation matrix in the MVO model is provided by Schafer et al. \cite{Schafer_etal_2010}. The basic idea is that the bulk of eigenvalues of the covariance matrix that are small (or even zero) may produce portfolios of stocks that have nonzero returns but extremely low or even vanishing risk; such portfolios are related to estimation errors in real data. For this reason, in \cite{Schafer_etal_2010} the authors propose two filtering techniques in order to eliminate the problem of small eigenvalues in the sampled covariance matrix. In the first technique, taken from \emph{Random Matrix Theory} (RMT) \cite{laloux_etal_2000}, after diagonalization of the correlation matrix, only the $p$ highest eigenvalues are preserved, while the other eigenvalues are set to zero. Then, the new filtered covariance matrix is obtained by using the filtered spectrum and the original eigenvectors. The second method, called \emph{Power Mapping}, transforms each element of the correlation matrix and raises its absolute value to some power $q$, while preserving the sign, thus obtaining a new matrix as input for the MVO model. The authors observe that one must pay attention to the value of $q$ that should be a function of the ratio $\frac{n}{T}$, where $n$ is the number of assets and $T$ the number of returns observations, i.e., the time horizon.
The first method was also considered in a series of previous papers \cite{PafkaKondor2002,PafkaKondor2003,PafkaKondor2004} where the authors apply a filter to the covariance matrix based on the RMT approach and introduce a measure defined as the ratio between the risk associated to an empirical portfolio and the risk associated with the \emph{true} optimal portfolio. In the above mentioned papers, the model minimizes the return's variance without the constraint on the fixed expected return level. The true portfolio was formed for the special case both when the true covariance matrix is equal to the unit matrix, and when it is randomly generated. Hence, the performance of noise filtering techniques is tested in a controlled setting.

\noindent From a risk management standpoint, natural quantities to estimate are the \emph{realized} (or out-of-sample) return and risk of a portfolio. Actually, the novelty in \cite{Schafer_etal_2010} is that the authors compare the two proposed filtering techniques applied to the MVO model with and without short-sellings, both in a Monte-Carlo simulation framework and, for the first time, on a real dataset. However, as the authors write, in a simulation framework good realized return and risk are observed, while in a real setting and with short-selling not allowed the improvement of the two filtering methods are not so evident w.r.t. the non-filtered MVO model (see \cite{Schafer_etal_2010} pages 116-117 Section 5.2.2). Moreover, they note experimentally that the power method approach is quite sensitive to $q$, and thus to the ratio $\frac{n}{T}$, and they find that a good value for $q$ should be approximatively equal to $\sqrt{\frac{n}{T}}$. The poor behavior of the RMT approach, in particular for the no short-selling case, is also confirmed in the paper \cite{ElKaroui2013}. In fact, the author observes that, in this case, naive estimates of risk are very close to realized risk, and no short-selling mitigates the problems due to biased sampled covariance matrices.

However, it is crucial to have good estimates for the correlation between stocks to reduce errors in the estimation of the portfolio return and risk. In this regard, a number of alternative approaches emerged in the literature in order to suppress noise in the data. These procedures include single- and multi-factor models, (for a review, see, e.g., \cite{EltGru1995} and \cite{EltGruSpi2006}) and Bayesian estimators (see \cite{Jorion1986,LedoitWolf2003}). Alternatively, some authors propose to construct a portfolio by
solving the classical Markowitz model in which the original correlation matrix is replaced by a correlation based clustering ultrametric matrix with the aim of providing portfolios that are quite robust with respect to measurement noise due to the finiteness of the sample size (see, among others, \cite{Onnela_et_al_2004,PuerMadSco2020,Tola_et_al_2008,WattsStrogatz_1998} and the references therein).
In the financial literature, different mathematical models have also been proposed in order to mitigate the effects of extreme events, which are based on the minimization of \emph{downside risk} measures. These measures are, in fact, functions of skeweness and kurtosis of the returns distributions. CVaR \cite{RockafellarUryasev_2002} is an example of such downside risk measures since it takes into account the entire part of the tail that is being observed.

To summarize, at the moment, although numerous alternatives to the MVO model have appeared in the literature in order to cope with biased return distributions when managing historical data, no definitive clear leader emerged up to now. Hence, the MVO model is still one of the most used practitioner-oriented tool.

\vskip 8 pt
\noindent In this paper we put ourselves from the practitioner point of view with the aim of applying the MVO model to obtain a portfolio from real financial datasets with a good out-of-sample performance. We propose a new filtering technique based on Combinatorial Optimization. In \cite{Stoyanov_etal_2011} the authors state that there is no fundamental theory that can suggest a distributional model for extreme events and the problem remains largely a statistical one. However, a financial market is characterized by some \emph{stylized facts} that, among others, are \cite{Stoyanov_etal_2011}: i) large price changes tend to be followed by large price changes and small price changes tend to be followed by small price changes; ii) price changes depend on price changes in the past, e.g. positive price changes tend to be followed by positive price changes.
These facts characterizing the markets can be exploited when selecting a portfolio. In particular, since due to shock periods, extreme events are unpredictable and few past extreme events can affect the portfolio out-of-sample performance too much, our approach is neglecting some (extremely rare) observations of the returns distributions in order to eliminate their influence on the performance of a portfolio in the future. Consider for example a long-only fund manager, the basic idea is that these shocks observed in the past can be considered as \emph{outliers} that can be eliminated since, in the long period, markets tend to replicate their typical behavior.

\noindent In this paper we propose a refinement of the MVO model that incorporates some additional variables that allows the model to eliminate observations (outliers) if this produces a variance decrease. The idea is that portfolios with good out-of-sample performance can be obtained in this way.

\vskip 8 pt
\noindent We compare the out-of-sample performance of our model to the two filtered MVO models presented in \cite{Schafer_etal_2010} testing them on some real financial datasets. We apply a rolling time window scheme and evaluate the future performance of our portfolios with the filtered ones not only in terms of realized return and risk, but also by considering the Sharpe ratio, which is a commonly used performance measure for comparisons purposes. The results are encouraging and, with our model, we obtain better performance results than those of the two filtering methods. This in spite of the computational effort required by our model which is a Mixed Integer Quadratic Programming (MIQP) problem. We show that standard solvers are able to find solutions in reasonable times at least for small and medium size financial dataset. For larger financial datasets we can still apply our approach by proposing a heuristic procedure that, from a computational viewpoint, have proven to be very efficient in practice.

\noindent The paper is organized as follows. Section 2 describes the two filtered MVO models presented in the literature based on the Random Matrix Theory and the Power Mapping techniques. In Section 3 we introduce our Scenario Filtering approach and formulate our general filtered MVO model as a nonlinear programming problem. We also discuss some complexity issues of this problem. In Section 3.1 we show how to reformulate it as a MIQP problem easy to solve with standard optimization software at least for small and medium size financial datasets. Also in this section, we derive a set of valid inequalities for the quadratic problem. A heuristic approach for solving large size instances is proposed in Section 3.2. Section 4 presents an extensive experimental analysis where we compare the out-of-sample performance of all the approaches on four real-world financial datasets. Finally, conclusions are depicted.

\section{The MVO filtering models: Notation and definitions}
\label{s:1}

The Mean-Variance Markowitz problem \cite{Markowitz1952} is a classic portfolio optimization problem in finance, where, given an amount of money, the aim is to select a portfolio of $n$ assets through two criteria: the expected return and the risk due to the variability of the returns. In the standard framework, the risk is measured by the variance of the portfolio returns. In the ideal case, the covariance and the mean of the returns are known, and the problem is: finding the proportion $x_j$ of capital invested in each asset $j$, $j=1,\ldots,n$, in a stock market in order to minimize the variance of the portfolio for a required level of expected return $\mu_0$. Let $\mu_j$ be the expected return of asset $j$, and $\sigma_{ij}$ be the covariance of returns of asset $i$ and asset $j$. The Markowitz model is formulated as the following convex quadratic program:

\begin{align}
\min & \,\,\,\,\, \sum_{i=1}^n \sum_{j=1}^n x_i x_j  \sigma_{ij}  \nonumber \label{MVO} \tag{\Mr{Markowitz}}\\
\mbox{ s.t. } \nonumber\\
&  \,\,\,\,\, \sum_{j=1}^n \mu_j x_j\geq  \mu_0  \noindent \nonumber\\
&  \,\,\,\,\, \sum_{j=1}^n x_j = 1  \noindent \nonumber\\
&  \,\,\,\,\, x_j\geq 0 \quad j=1,\ldots,n.   \noindent \nonumber
\end{align}

\noindent The above model refers to the case where short-selling is not allowed in the optimization, i.e. the fractions $x_j$ must be nonnegative. As observed in \cite{Schafer_etal_2010}, imposing this constraint has the direct and rather limiting consequence that positive correlations between assets' returns cannot be used to reduce the portfolio risk. As we will see later, this may limit the potential for lowering the risk through the application of noise reduction methods based on spectral filtering. When short-selling is allowed, the fractions $x_j$ are not restricted to be nonnegative. In any case, we observe that short-selling models are considered rather unrealistic in the specialized literature, as pointed out in \cite{Kondor_etal_2007}, for legal and liquidity reasons.

\vskip 8 pt

\noindent In a real financial market, we have $T+1$ different observations (scenarios) for the prices of $n$ given assets. Let $P_{jt}$ be the price of asset $j$, $j=1,\ldots,n$, at time $t$, $t=0,1,\ldots,T$. For a financial risk manager it is of crucial interest to have good estimates for the returns and correlations between stocks. Following \cite{Schafer_etal_2010}, we compute the $T$ rates of return of asset $j$, $r_{jt}$, as

$$r_{jt}=\frac{P_{jt}-P_{jt-1}}{P_{jt-1}},\quad j=1,\ldots,n,\quad t=1,\ldots,T.$$

\noindent We assume that, for each asset $j$, the observed return $r_{jt}$ has associated a probability $p_t$, $t=1,\ldots,T$, with $\sum_{t=1}^T p_t =1$. Under the hypothesis of no further information, we assume $p_t=\frac{1}{T}$, $t=1\ldots,T$. The average rate of return of asset $j$ is $\mu_j=\sum_{t=1}^T p_t r_{jt}$, $j=1,\ldots,n$. We estimate the covariance $\sigma_{ij}$ between the rate of returns of assets $i$ and $j$ by computing

$$\sigma_{ij}=\frac{1}{T}\sum_{t=1}^T (r_{it}-\mu_i)(r_{jt}-\mu_j), \quad i,j=1,\ldots,n.$$

\noindent The correlation matrix $C$ is given by $C_{ij}=\frac{\sigma_{ij}}{\sigma_{i}\sigma_{j}}$, $i,j=1,\ldots,n$, where the volatility of asset $j$ is measured by $\sigma_{j}=+\sqrt{\sigma_{jj}}$, $j=1,\ldots,n$. The rate of return at time $t$ of a portfolio $x=(x_1,\ldots,x_n)$ is

$$ y_t(x) = \sum_{j=1}^n r_{jt} x_j,\quad\quad t=1,\ldots,T,$$

\noindent the expected portfolio rate of return is

$$\mu(x) = \sum_{t=1}^T p_t y_t(x) = \sum_{j=1}^n \mu_{j} x_j,$$

\noindent and the portfolio risk can be measured by

$$V(x) = \sum_{i=1}^n \sum_{j=1}^n x_i x_j  \sigma_{ij}  = \sum_{i=1}^n \sum_{j=1}^n x_i x_j  \sigma_i \sigma_j C_{ij}.$$

\noindent It is important to note that the correlation matrix depends on the time horizon $T$ and, for a finite $T$, the correlations obtained from historical data are affected by a considerable amount of noise, which leads to a substantial error in the estimation of the portfolio risk \cite{Fama1965,JansenVries_1991,Mandelbrot1963,Schafer_etal_2010}.

\vskip 8 pt
\noindent The Random Matrix Theory introduced by \cite{laloux_etal_2000} is helpful to identify the noise in correlation matrices, and also shows a way to reduce this noise. Actually, the idea is that, after diagonalization of the correlation matrix $C = U^{-1}\Lambda U$, only the $p$ highest eigenvalues in the diagonal of $\Lambda$ must be considered, while the remaining ones are set to zero. Let $\Lambda^f$ be the filtered eigenvalues diagonal matrix and $U$ be the original eigenvector matrix. Then, the new filtered correlation matrix $C^f$ is computed

$$C^f=U^{-1}\Lambda^fU.$$

\noindent The normalization of the elements on the diagonal of $C^f$ to 1 is restored by setting $C_{jj}^f=1$ for all $j=1,\ldots,n$. This method is capable of removing the noise for uncorrelated assets completely (see \cite{Schafer_etal_2010} for a complete description of the method). Finally, the MVO model filtered according to the RMT becomes:

\begin{align}
\min & \,\,\,\,\, \sum_{i=1}^n \sum_{j=1}^n x_i x_j\sigma_{i}\sigma_{j} C_{ij}^f \nonumber \label{RMT-MVO} \tag{RMT}\\
\mbox{ s.t. }\nonumber \\
&  \,\,\,\,\, \sum_{j=1}^n \mu_j x_j\geq  \mu_0  \noindent \nonumber\\
&  \,\,\,\,\, \sum_{j=1}^n x_j = 1  \noindent \nonumber\\
&  \,\,\,\,\, x_j\geq 0 \quad j=1,\ldots,n.   \noindent \nonumber
\end{align}

\vskip 8 pt
\noindent The weakness of the RMT method is that it cuts off all the information contained in the discarded eigenvalues. An alternative  method developed in order to reduce noise was introduced in \cite{GuhrKal_2003}. It is the so-called Power Mapping technique. It takes each element of the correlation matrix and raises its absolute value to some power $q$, while preserving the sign. With this method one obtains a new correlation matrix whose elements are:

$$C_{ij}^{(q)}=\text{\Mr{sign}}(C_{ij})|C_{ij}|^q.$$

\noindent The idea behind this method is that the effect of the noise, which typically arises in the small correlations, can be broken in this way, with an effect which is stronger as the value of $q$ increases. Thus, the problem is to choose the right value for $q$ also taking into account that a byproduct effect of $q$ is produced an all correlations. The MVO model filtered according to the Power Mapping is then formulated as follows:

\MR{
\begin{align}
\min & \,\,\,\,\, \sum_{i=1}^n \sum_{j=1}^n x_i x_j\sigma_{i}\sigma_{j} C_{ij}^{(q)} \nonumber \label{PM-MVO} \tag{Power Mapping}\\
\mbox{ s.t. }\nonumber \\
&  \,\,\,\,\, \sum_{j=1}^n \mu_j x_j\geq  \mu_0  \noindent \nonumber\\
&  \,\,\,\,\, \sum_{j=1}^n x_j = 1  \noindent \nonumber\\
&  \,\,\,\,\, x_j\geq 0 \quad j=1,\ldots,n.   \noindent \nonumber
\end{align}
}

\vskip 12 pt
\noindent The drawback of the two above methods is that, generally, to reduce noise, they, in fact, eliminate too much information contained in the observed time series of the assets prices, and this inevitably affects the estimation of the portfolio return and risk.
\vskip 8 pt

\section{The Scenario Filtering approach}
\label{s:2}

\noindent In this paper we propose an alternative refinement of the MVO model that incorporates some additional constraints that allows the model to eliminate observations (outliers) in order to obtain a good estimation of return and risk. The idea is that, in particular in the long run, few extreme events in the distribution of assets prices could extremely affect the volatility and performance of a portfolio in the future. With our method we maintain the bulk of the data while dropping only those extreme observations caused by a distribution with fat tails or more simply by measurement errors.


\vskip 8 pt
\MR{

\noindent More precisely, we want to model the problem of filtering (eliminating) a fixed number $K$ of the $T$ observed scenarios for the assets rate of returns, $K < T$, while simultaneously solving the MVO problem. \mrd{We refer to this problem as Problem (P).} The selection of the filtered scenarios can be modeled with binary variables:

$$z_{t} = \begin{cases} 1 \qquad \hbox{if scenario $t$ is filtered}\\ 0 \qquad \hbox{otherwise}\end{cases} \quad t=1,\ldots, T.$$

\noindent Under the equiprobability assumption, for a filtering realization $z=(z_1,\ldots,z_T)$ of $K<T$ scenarios, the return distribution of a portfolio $x=(x_1,\ldots,x_n)$ is $y_t(x)$ with probability $\tilde{q}_t(z_t)=\frac{1}{T-K}(1-z_t)$. Then, for each $t=1,\ldots,T$ we denote by $q_t = \frac{1}{T-K}$ and, recalling that $y_t(x)=\sum\limits_{j=1}^nr_{jt}x_j$, the \emph{filtered expected return} of the portfolio is

\begin{equation}\label{mutilde}
\tilde{\mu}(x,z) = \sum_{t=1}^T \tilde{q}_t(z_t) y_t(x) = \sum_{t=1}^T \sum_{j=1}^n q_t r_{jt} x_j (1-z_t),
\end{equation}

\noindent and the corresponding \emph{filtered variance} is

\begin{equation}\label{Vtilde}
\tilde{V}(x,z) = \sum_{t=1}^T \tilde{q}_t(z_t) (y_t(x)-\tilde{\mu}(x,z))^2 = \sum_{t=1}^T q_t (y_t(x)-\tilde{\mu}(x,z))^2 (1-z_t).
\end{equation}

\noindent \mrd{Problem (P)} can be stated as:

\begin{align}
\min & \,\,\,\,\, \tilde{V}(x,z) \nonumber \label{P} \tag{P}\\
\mbox{ s.t. } \nonumber\\
&  \,\,\,\,\, \tilde{\mu}(x,z)\geq  \mu_0  \noindent \nonumber\\
&  \,\,\,\,\, \sum_{j=1}^n x_j = 1  \noindent \nonumber\\
&  \,\,\,\,\, \sum_{t=1}^T z_t = K  \noindent \nonumber\\
&  \,\,\,\,\, x_j\geq 0 \quad j=1,\ldots,n   \noindent \nonumber\\
&  \,\,\,\,\, z_t\in\{0,1\} \quad t=1,\ldots,T. \noindent \nonumber
\end{align}

}

\noindent Problem (\ref{P}) is a Quadratically Constrained Mixed Integer Nonlinear Programming problem that falls into the
class of considerably difficult NP-hard problems \cite{FlouVis_1995}. We also show that testing feasibility of problem (\ref{P}) is NP-hard in two cases: (i) when we require that the portfolio expected return be exactly equal to $\mu_0$ (see e.g. \cite{CeScoTar_2013,Schafer_etal_2010}); (ii) when we add \emph{threshold constraints} on investment $\ell_j \leq x_j \leq u_j$, $j=1,\ldots,n$, which impose limitations to the assets shares in the portfolio (see \cite{ManOgrySpe_2007}). Indeed, we note that  in this second case we only need to assume upper bounds $u_j > 0$, $j=1,\ldots,n$.

We prove these complexity results via a polynomial reduction to the \emph{Partition} problem that can be stated as follows: Given a finite set $A = \{a_1,\ldots,a_n\}$ with $a_i\in \mathbb{Z}_+$ for all $i\in I=\{1,\ldots,n\}$, the Partition problem asks for the existence of an index subset $I'\subseteq I$ such that 
$$\sum_{i\in I'} a_i = \sum_{i\in I\setminus I'}a_i.$$

\noindent The Partition problem is NP-complete and remains NP-complete even if we require $|I'|=\frac{n}{2}$ (see \cite{GareyJohnson1979}).

\mrm{
\begin{prop}\label{NP-hardness1}
Testing feasibility of \mrd{Problem \eqref{P}} is NP-hard when $\tilde{\mu}(x,z)=\mu_0$ is required.
\end{prop}

\begin{proof}
In Problem \eqref{P}, set $T=n$, $K=\frac{n}{2}$, $\mu_0=\frac{1}{n}\sum_{i=1}^n a_i$, and the assets rate of returns $r_{jt} = a_t$ for all $j,t\in I=\{1,\ldots,n\}$. Then, the rate of return at time $t\in I$ of any feasible portfolio $x$ is
$$ y_t(x) = \sum_{j=1}^n a_{t} x_j = a_t,$$
\noindent and the corresponding filtered expected return becomes
$$\tilde{\mu}(x,z)  = \frac{2}{n} \sum_{t=1}^T a_t (1-z_t).$$
\vskip 8 pt
\noindent In Problem \eqref{P} we assume that, given a feasible solution, if $z_t=1$ then $t\in I'\subset I$, otherwise $t\in I\setminus I'$. Furthermore, we also observe that when in the Partition problem $|I'|=\frac{n}{2}$ is required, it implies that

$$\sum_{i\in I'} a_i = \sum_{i\in I\setminus I'}a_i=\frac{\sum_{i\in I} a_i}{2}.$$ 

\noindent Hence, the answer to the Partition problem is ``yes'' if and only if Problem \eqref{P} is feasible.
\qed
\end{proof}

\begin{prop}\label{NP-hardness2}
Testing feasibility of \mrd{Problem \eqref{P}} is NP-hard when  $x_j\leq u_j$, with $u_j>0$, $j=1,\ldots,n$, is required.
\end{prop}

\begin{proof}
We refer to a modified Partition problem, that is, we assume $A\subseteq \mathbb{Q}_+$ and $\sum_{i=1}^n a_i = 2$. This assumption implies that if $a_i>1$ for some $i\in I$, then the answer to the problem is trivially ``no''. Thus, we assume $a_i\leq 1$ for all $i\in I$. We observe that this modified problem is equivalent to the (standard) Partition problem if we multiply the numbers in the set $A$ by a proper factor. Hence, the modified Partition problem is NP-complete, as well.

\noindent In Problem \eqref{P}, set $T=n$, $K=\frac{n}{2}$, $\mu_0=\frac{n}{2}$, $u_j = a_j$ for all $j=1,\ldots,n$, and the assets rate of returns $r_{jt} = \frac{1}{a_j}$ if $t\neq j$, and $r_{jt} = -(K-1)\frac{1}{a_j}-K\varepsilon$ if $t = j$, for all $j,t\in I=\{1,\ldots,n\}$ and any $\varepsilon > 0$. Then, the filtered average rate of return of asset $j\in I$ is

$$\tilde{\mu}_j(z)=\frac{1}{K}\sum_{t=1}^T r_{jt} (1-z_t).$$

Note that $\tilde{\mu}_j(z)=\frac{1}{a_j}$ if $z_j = 1$, and $\tilde{\mu}_j(z)=-\varepsilon$ if $z_j = 0$. Thus, it follows that
$\tilde{\mu}_j(z)= \frac{1}{a_j} z_j -\varepsilon (1-z_j)$ for all $j\in I$. Hence, the filtered expected return of the portfolio can be expressed as

$$\tilde{\mu}(x,z) = \sum_{j=1}^n \tilde{\mu}_j(z) x_j = \sum_{j=1}^n  \frac{1}{a_j} z_j x_j -\varepsilon \sum_{j=1}^n  (1-z_j) x_j.$$

Finally, the feasible region of Problem \eqref{P} becomes
\begin{align}
&  \,\,\,\,\, \sum_{j=1}^n  \frac{1}{a_j} z_j x_j -\varepsilon \sum_{j=1}^n  (1-z_j) x_j\geq  K  \noindent \nonumber\\
&  \,\,\,\,\, \sum_{j=1}^n x_j = 1  \noindent \nonumber\\
&  \,\,\,\,\, \sum_{j=1}^n z_j = K  \noindent \nonumber\\
&  \,\,\,\,\, 0\leq x_j\leq a_j \quad j=1,\ldots,n   \noindent \nonumber\\
&  \,\,\,\,\, z_j\in\{0,1\} \quad j=1,\ldots,n. \noindent \nonumber
\end{align}

\noindent Therefore, with the same reasoning of the previous proposition, the answer to the (modified) Partition problem is ``yes'' if and only if Problem \eqref{P} is feasible.
\qed
\end{proof}
}

\noindent In the light of the above propositions, it is important to provide efficient solution methods to solve \mrd{Problem \eqref{P}}. In the following, we propose both an exact method and a heuristic procedure.


\subsection{The Mixed Integer Quadratic Programming model}
\label{s:21}

\noindent In this Section we present a reformulation of Problem \eqref{P} as a MIQP problem that, even if it still belongs to the class of difficult NP-hard problems, it can be efficiently solved by using some commercial or free optimization solvers for general MIQP models.

\vskip 8 pt
\noindent We consider the filtered variance of the portfolio return \eqref{Vtilde}:

$$ \tilde{V}(x,z) = \sum_{t=1}^T q_t |y_t(x)-\tilde{\mu}(x,z)|^2 (1-z_t). $$

\noindent Since in Problem (\ref{P}) we minimize the filtered variance we can apply the McCormick linearization \cite{McCormick1976} to the products of squared absolute values and $(1-z_t)$ terms. Therefore, the filtered variance can be equivalently written as

\begin{equation}\label{Var_d}
\sum_{t=1}^T q_t d_t^2
\end{equation}

\noindent whenever the following set of constraints is satisfied

\begin{align}
&  \,\,\,\,\, d_{t}\geq |y_t(x)-\tilde{\mu}(x,z)|-z_t M_t \quad t=1,\ldots,T   \noindent \nonumber
\end{align}

\noindent for nonnegative variables $d_t\geq 0$ and big enough constants $M_t>0$, $t=1,\ldots,T$. \Mr{In addition,} since $|y_t(x)-\tilde{\mu}(x,z)|= \max\{y_t(x)-\tilde{\mu}(x,z), -(y_t(x)-\tilde{\mu}(x,z))\}$, the above set of constraints can be replaced by the following one:

\begin{equation}\label{Const_d}
\begin{array}{lll}
& d_{t}\geq y_t(x)-\tilde{\mu}(x,z)-z_t M_t^{+} \quad t=1,\ldots,T
\\
& d_{t}\geq -(y_t(x)-\tilde{\mu}(x,z))-z_t M_t^{-} \quad t=1,\ldots,T
\end{array}
\end{equation}

\noindent where the constants $M_t^+,M_t^->0$ are big enough for each $t=1,\ldots,T$. We can also linearize the quadratic terms $x_j(1-z_t)$ in $\tilde{\mu}(x,z)$ by adding a new set of variables $\tilde{x}_{jt}\geq 0$, $j=1,\ldots,n$, $t=1,\ldots,T$, such that $\tilde{x}_{jt}=x_j (1-z_t)$ by imposing the following constraints:

\begin{equation}\label{Const2_d}
\begin{array}{lll}
&  \tilde{x}_{jt}\leq (1-z_t) \quad j=1,\ldots,n,t=1,\ldots,T \\
&  \tilde{x}_{jt}\leq x_j \quad j=1,\ldots,n,t=1,\ldots,T   \\
&  \tilde{x}_{jt}\geq x_j-z_t \quad j=1,\ldots,n,t=1,\ldots,T.
\end{array}
\end{equation}

\noindent Constraints (\ref{Const2_d}) refer to a standard way of linearizing the product of a binary variable and a continuous nonnegative one. They can be included, together with (\ref{Const_d}), \mrd{in a reformulation of Problem \eqref{P} in which the objective function (\ref{Var_d}) is minimized.}

It is easy to see that, when integrated in our model, constraints (\ref{Const2_d}) are implied by the following ones:

\begin{equation}\label{Const3_d}
\begin{array}{lll}
& \displaystyle \sum_{j=1}^n \tilde{x}_{jt} =  1-z_t \quad t=1,\ldots,T 
\\
& \tilde{x}_{jt}\leq x_j \quad j=1,\ldots,n,t=1,\ldots,T.
\end{array}
\end{equation}

\noindent The above constraints lead to an equivalent formulation that should be \emph{stronger}. This, in fact, has been confirmed in our computational experiments. Thus, in our formulation we replace (\ref{Const2_d}) by (\ref{Const3_d}).

\MR{

\vskip 8 pt
\noindent From the above discussion, we propose the following \mrd{MIQP model}:

\begin{align}
\min & \,\,\,\,\, \sum_{t=1}^T q_t d_t^2 \nonumber \label{K-SF-MVO} \tag{\mrd{Scenario Filtering}}\\
\mbox{ s.t. } \nonumber\\
&  \,\,\,\,\, \Mr{d_{t'}\geq \sum_{j=1}^n r_{jt'} x_j-\sum_{t=1}^T \sum_{j=1}^n q_t r_{jt} \tilde{x}_{jt}-z_{t'} M_{t'}^{+} \quad t'=1,\ldots,T}   \noindent \nonumber\\
& \,\,\,\,\, \Mr{d_{t'}\geq -\sum_{j=1}^n r_{jt'} x_j+\sum_{t=1}^T \sum_{j=1}^n q_t r_{jt} \tilde{x}_{jt}-z_{t'} M_{t'}^{-} \quad t'=1,\ldots,T}     \noindent \nonumber\\
&  \,\,\,\,\, \sum_{j=1}^n \tilde{x}_{jt} =  1-z_t \quad t=1,\ldots,T   \noindent \nonumber\\
&  \,\,\,\,\, \Mr{\tilde{x}_{jt}\leq x_j \quad j=1,\ldots,n,t=1,\ldots,T   \noindent \nonumber}\\
&  \,\,\,\,\, \sum_{t=1}^T \sum_{j=1}^n q_t r_{jt} \tilde{x}_{jt}\geq  \mu_0  \noindent \nonumber\\
&  \,\,\,\,\, \sum_{j=1}^n x_j = 1  \noindent \nonumber\\
&  \,\,\,\,\, \sum_{t=1}^T z_t = K  \noindent \nonumber\\
&  \,\,\,\,\, x_j\geq 0 \quad j=1,\ldots,n   \noindent \nonumber\\
&  \,\,\,\,\, z_t\in\{0,1\} \quad t=1,\ldots,T \noindent \nonumber\\
&  \,\,\,\,\, d_t\geq 0 \quad t=1,\ldots,T \noindent \nonumber\\
&  \,\,\,\,\, \tilde{x}_{jt}\geq 0 \quad j=1,\ldots,n,t=1,\ldots,T. \noindent \nonumber
\end{align}

\noindent The above model is a \MRM{MIQP model} that uses big-$M$ constants. It is well-known that big-$M$ constants often produce large gaps between the continuous relaxation of the model and the MIP objective values, which can induce a poor performance of the model from the computational time point of view. Thus, it is important to provide \emph{tight} values for these constants. In our case, for a given $t'=1,\ldots,T$, we can set the values as follows:

\Mr{$$ M_{t'}^{+} = \max_{\scriptsize\begin{array}{c}x_1,\ldots,x_n\geq 0\\ z_1,\ldots,z_T\in \{0,1\}\end{array}}\left\{\sum_{j=1}^n x_j \left( r_{jt'}- \sum_{t=1}^T q_t r_{jt}(1-z_t) \right) : \sum_{j=1}^n x_j = 1, \sum_{t=1}^T z_t = K , z_{t'}=1\right\} $$}

\noindent and

\Mr{$$ M_{t'}^{-} = \max_{\scriptsize\begin{array}{c}x_1,\ldots,x_n\geq 0\\ z_1,\ldots,z_T\in \{0,1\}\end{array}}\left\{\sum_{j=1}^n x_j \left( -r_{jt'}+ \sum_{t=1}^T q_t r_{jt}(1-z_t) \right) : \sum_{j=1}^n x_j = 1, \sum_{t=1}^T z_t = K , z_{t'}=1\right\}. $$}

\noindent It is not difficult to see that the solution of the above maximization problems can be easily computed as

\Mr{$$ M_{t'}^{+} = \max_{j=1,\ldots,n}\left\{r_{jt'} + B_{jt'}^+\right\} \quad \text{ and } \quad M_{t'}^{-} = \min_{j=1,\ldots,n}\left\{-r_{jt'} + B_{jt'}^- \right\}$$}

\noindent  \Mr{where $B_{jt'}^+$ is the sum of the $K-1$ greater numbers of the set $\{-q_1 r_{j1},\ldots,-q_T r_{jT}\}\setminus\{-q_{t'} r_{jt'}\}$ and $B_{jt'}^-$ is the sum of the $K-1$ greater numbers of the set $\{q_1 r_{j1},\ldots,q_T r_{jT}\}\setminus\{q_{t'} r_{jt'}\}$, for each $j=1,\ldots,n$.}

}

\vskip 8 pt

\noindent We observe that, although alternative MIQP models without big-$M$ constants can be obtained by using different reformulations of Problem \eqref{P}, the advantage of our \mrd{MIQP model} reformulation is that its objective function is convex. Due to this property, the continuous relaxation of the problems arising from a Branch and Bound strategy are convex quadratic programs, for which solvers provide more reliable solutions than the ones obtained for nonconvex problems. Reliability of solutions is crucial when one takes into account the numerical issues that occur when optimizing a MIQP model (see \cite{Bienstock1996}), and also when one consider the possible effect of these numerical issues on the portfolios (out-of-sample) performance. In this way, as observed in our computational experiments, \mrd{our MIQP model} guarantees the reliability of the solutions found without worsening the computational times too much. This is due to the tightness of the big-$M$ constants. Moreover, we observed that when additional variables are used to reformulate \mrd{our MIQP model} avoiding the use of big-$M$ constants, the computational time saved does not compensate for the additional time obtained by augmenting the problem's dimension.

\vskip 8 pt

\noindent \mrd{Our MIQP model} admits a set of valid inequalities related with some others presented in the literature for similar problems. In particular, we note the similarity of Problem \eqref{P} with the problem studied in \cite{Bienstock1996}. In the historical data approach assumed in the present work, the set of considered stocks and the set of available observations determine the input for the MVO problem. Roughly speaking, the problem in \cite{Bienstock1996} corresponds to solving the MVO problem by selecting a cardinality constrained subset of the stocks. In our case, we are interested in solving the  MVO problem but we want to select a subset of time-observations. Therefore, in the former problem there is a cardinality constraint over the number of assets included in the portfolio, while in our model the cardinality constraint is on the number of considered observations. Of course, the constraints are similar but not the same. As in \cite{Bienstock1996} we derive a set of valid inequalities for our problem that are inspired by the cover inequalities for the knapsack problem \cite{NemhauserWolsey1988,WolseyInteger1998}.

\vskip 8 pt

\noindent We have derived \mrd{our MIQP model} by reformulating Problem \eqref{P}. In that process, we have introduced variables $\tilde{x}_{jt}=x_j(1-z_t)$, $j=1,\ldots,n,t=1,\ldots,T$, originally not defined in Problem \eqref{P}. We show now that indeed Problem \eqref{P} can be conceptually stated using only $\tilde{x}_{jt}$ variables. We start by noting that the \MRM{filtered expected return} \eqref{mutilde} can be expressed using $\tilde{x}_{jt}$ variables as

\begin{equation*}
\tilde{\mu}(\tilde{x}) = \sum_{t=1}^T \sum_{j=1}^n q_t r_{jt} \tilde{x}_{jt}.
\end{equation*}

In the same way, and taking into account that $\tilde{V}(x,z) = \sum_{t=1}^T \tilde{q}_t(z_t) (y_t(x))^2 - (\tilde{\mu}(x,z))^2$, we can write the corresponding filtered variance \eqref{Vtilde} as

\begin{equation*}
\tilde{V}(\tilde{x}) = \sum_{t=1}^T q_t(\sum_{j=1}^n  r_{jt} \tilde{x}_{jt})^2 - (\tilde{\mu}(\tilde{x}))^2.
\end{equation*}

\mr{We also need to rewrite the minimum expected return requirement constraint  
\begin{equation*}
\sum_{t=1}^T \sum_{j=1}^n q_t r_{jt} \tilde{x}_{jt}\geq \mu_0
\end{equation*}
in the equivalent form with positive coefficients and positive right-hand side
\begin{equation*}
\sum_{t=1}^T \sum_{j=1}^n a_{jt} \tilde{x}_{jt}\geq \beta,
\end{equation*}
where $a_{jt} = q_t (r_{jt} + \alpha) > 0$, $j=1,\ldots,n,t=1,\ldots,T$, $\beta = \mu_0+\alpha > 0$, and $\alpha\geq 0$ is taken big enough. The equivalence between the constraints above follows from the fact that 
$\sum_{t=1}^T \sum_{j=1}^n q_t \alpha \tilde{x}_{jt} = \alpha\frac{1}{T-K}\sum_{t=1}^T (1-z_t)\sum_{j=1}^n x_j = \alpha $ in the setting of Problem \eqref{P}.} Finally, Problem \eqref{P} can be re-stated as

\mr{
\begin{align}
\min & \,\,\,\,\, \tilde{V}(\tilde{x}) \nonumber \label{P_c} \tag{P$_{\tilde{x}}$}\\
\mbox{ s.t. } \nonumber\\
&  \,\,\,\,\, \sum_{t=1}^T \sum_{j=1}^n a_{jt} \tilde{x}_{jt}\geq  \beta  \noindent \nonumber\\
&  \,\,\,\,\, \left|\left\{t\in \{1,...,T\}: \tilde{x}_{jt}>0 \text{ for some } j=1,\ldots,n\right\}\right| = T-K  \noindent \nonumber\\
&  \,\,\,\,\, \sum_{j=1}^n \tilde{x}_{jt} = 1 \quad t=1,\ldots,T:\tilde{x}_{jt}>0 \text{ for some } j=1,\ldots,n   \noindent \nonumber\\
&  \,\,\,\,\, \tilde{x}_{jt} = \tilde{x}_{jt'} \quad j=1,\ldots,n,t,t'=1,\ldots,T: x_{jt},x_{jt'}>0  \noindent \nonumber\\
&  \,\,\,\,\, 0\leq \tilde{x}_{jt}\leq 1 \quad j=1,\ldots,n,t=1,\ldots,T.   \noindent \nonumber
\end{align}

\vskip 8 pt

\noindent Using the same terminology as in \cite{Bienstock1996}, we say that a set $S\subseteq\{1,...,T\}$ is \emph{critical} if: 
\begin{equation*}
\text{for every $R\subseteq S$ with $|R|=T-K$, \,\,$\sum_{t\in R}\sum_{j=1}^n a_{jt} < \beta$.}
\end{equation*}
Then, the following results hold.
}

\begin{prop}\label{PropValidInequalities}
If $S$ is a critical set, then the set of inequalities
\begin{equation*}
\sum_{t\in S}\tilde{x}_{jt} \leq T-K-1 \quad j=1,...,n
\end{equation*}
are valid for Problem \eqref{P_c}.
\end{prop}

\begin{proof}
Note that the definition of critical set $S$ implies that in Problem \eqref{P_c} at most $T-K-1$ indices $t\in S$ can have associated nonzero $x_{jt}$ variables, otherwise the minimum expected return requirement constraint is not satisfied. Since Problem \eqref{P_c} also forces the consistency of the `actives' scenarios in the sets $\{\tilde{x}_{j1},\ldots,\tilde{x}_{jT}\}$, $j=1,\ldots,n$, the result follows.\qed
\end{proof}

\begin{cor}\label{CorValidInequalities}
The set of inequalities given in Proposition \ref{PropValidInequalities} are also valid for \mrd{our MIQP model}.
\end{cor}

\noindent We remark that each critical set in \cite{Bienstock1996}, or cover set in \cite{WolseyInteger1998} (also called dependent set in \cite{NemhauserWolsey1988}), has a naturally associated valid inequality for the problem considered in each reference. Given the particularity of our Problem \eqref{P_c}, each critical set provides $n$ valid inequalities for this problem. 

\vskip 8 pt

\noindent As we will see in the Experimental results section, \mrd{our MIQP model} is \Mr{efficiently solvable for small and medium size financial datasets}. For larger size instances, computational times are considerably longer. For this reason, in the following section we propose a heuristic approach for solving \mrd{Problem \eqref{P}}.

\subsection{\mrd{The heuristic algorithm}}

\mr{In this subsection we present our heuristic procedure. We note that solving \mrd{Problem \eqref{P}} implies to decide which $K$ scenarios have to be filtered from the $T$ observed ones. This corresponds to make a decision among $\binom{T}{K}$ possible ones. Thus, it is reasonable to think that the computational effort to solve \mrd{Problem \eqref{P}} grows with the cardinality of the decision set, from $K=1$ up to $K=\frac{T}{2}$. Indeed, this was confirmed by our computational experiments. Based on this consideration, we design a heuristic algorithm that exploits a nested solutions strategy. \mrd{Let \mrd{Problem (P$_k$)} be our Problem \eqref{P} when $k\leq K$ scenarios have to be filtered, and let $z(k)=(z_{1}(k),\ldots,z_{t}(k),\ldots, z_{T}(k))$ be the filtering decision variables of Problem (P$_{k}$)}, $k=1,\ldots,K$. Then, at each step we solve \mrd{Problem (P$_{k+1}$)} but keeping the best scenarios filtering $z^*(k)$ found by the heuristic for the previous \mrd{Problem (P$_k$)}, $k=1,\ldots,K-1$.  The computational effort to perform each step equals the effort for solving \mrd{Problem \eqref{P} when only one scenario has to be filtered}. We show the basic pseudocode below.}

\begin{algorithm}
\caption{\mrd{Heuristic}}
\label{alg:sequential}
\renewcommand{\thealgorithm}{}
\floatname{algorithm}{}
\begin{algorithmic}[1]
    \Begin
    \State solve \mrd{Problem (P$_1$)}
    \State \mr{let $z^*(1)$ be the scenarios filtering obtained}
    \For{$k=2,\ldots,K$}
    \State \Mr{solve \mrd{Problem (P$_k$)} fixing $z_{t}(k)=1$ if $z_{t}^*(k-1)=1$, $t=1,\ldots,T$}
    \State \mr{let $z^*(k)$ be the scenarios filtering obtained}
    \EndFor
    \State \Return $x^*$ the best portfolio found
    \End
\end{algorithmic}
\end{algorithm}

\vskip 8 pt

\noindent It is clear that, in Algorithm \ref{alg:sequential} the solution found at the first step (line 2) is optimal for $k=1$. On the contrary, from step 2 on, we obtain suboptimal solutions for $k\geq 2$. In spite of this, in our experiments the heuristic showed to be effective in finding good quality solutions. To implement Algorithm \ref{alg:sequential} we can solve \mrd{our MIQP model for $K=1$} in line 2, and \mrd{our MIQP model for $K=k$} in line 5 but fixing appropriately the $z_t$ variables. We call Version 1 this implementation of Algorithm \ref{alg:sequential}. An alternative implementation, which we call Version 2, is described below.

\vskip 8 pt

\noindent Let $\mathcal{T}=\{1,\ldots,T\}$. Suppose that we want to solve the MVO problem but considering only a given subset of the observations $R\subseteq \mathcal{T}$. From the results in Section 3.1, it is straightforward to see that such a problem can be formulated as follows:

\begin{align}
\min & \,\,\,\,\, \sum_{t\in R} \hat{p}_t(R)d_t^2 \nonumber \label{RMVO} \tag{$R$-MVO}\\
\mbox{ s.t. } \nonumber\\
&  \,\,\,\,\, d_{t'}\geq \sum_{j=1}^n r_{jt'} x_j-\sum_{t\in R} \sum_{j=1}^n \hat{p}_t(R) r_{jt} x_j \quad t'\in R   \noindent \nonumber\\
&\,\,\,\,\, d_{t'}\geq -\sum_{j=1}^n r_{jt'} x_j+\sum_{t\in R} \sum_{j=1}^n \hat{p}_t(R) r_{jt} x_j \quad t'\in R   \noindent \nonumber\\
&  \,\,\,\,\, \sum_{t\in R} \sum_{j=1}^n \hat{p}_t(R) r_{jt} x_j\geq  \mu_0  \noindent \nonumber\\
&  \,\,\,\,\, \sum_{j=1}^n x_j = 1  \noindent \nonumber\\
&  \,\,\,\,\, x_j\geq 0 \quad j=1,\ldots,n   \noindent \nonumber\\
&  \,\,\,\,\, d_t\geq 0 \quad t\in R \noindent \nonumber
\end{align}
where $\hat{p}_t(R)=\frac{1}{|R|}$, $t\in R$. Problem $R$-MVO is a simple convex quadratic program and therefore it can be solved quite efficiently. Now, note that line 2 in Algorithm \ref{alg:sequential} can be performed by solving Problem $R$-MVO for each set $R=\mathcal{T}\setminus\{t\}$, $t\in\mathcal{T}$, and returning the best solution found. In the case of line 5, it can be performed by solving Problem $R$-MVO for each set $R=(\mathcal{T}\setminus \{t\in \mathcal{T}: z_t^*(k-1)=1\})\setminus\{t'\}$, $t'\in\{t\in \mathcal{T}: z_t^*(k-1)=0\}$, and  returning the best solution found. This procedure is our Version 2 of the implementation of Algorithm \ref{alg:sequential}. The reader may notice that this procedure can be replicated using the original MVO model, which is also a convex quadratic program, using the corresponding covariance matrix in function of the considered observations. However, this alternative way of performing the procedure requires the computation of the covariance matrix for each MVO model, which charges an unnecessary additional computational effort.

\vskip 8 pt

\noindent  Which among the two versions of Algorithm \ref{alg:sequential} is more efficient depends on the number of observations $T$ and the size $n$ of the set of assets. In the conditions of our computational experiments, for a fixed $T$, Version 1 of the Algorithms performs better than Version 2 for small size datasets, while Version 2 performs quite better than Version 1 for medium/large size datasets.

\section{Experimental results}

In this section we present an empirical analysis on real stocks market data with the aim of evaluating the out-of-sample performance of the portfolios selected by \Mr{the MVO model when filtered according to our Scenario Filtering approach.}
In addition, we compare the performances of our filtered portfolios with those obtained by the MVO model filtered by adopting the RMT and the Power Mapping techniques, and the portfolios selected by the classical Markowitz model.
\vskip 8 pt

\noindent We test all the above portfolio selection strategies on some real-world datasets belonging to the major stock markets across the world. We consider the following datasets\footnote{The datasets are available upon request} that were also used in \cite{BruCeScoTar_2017,PuerMadSco2020}:

\begin{itemize}
\item[1.] DJIA (Dow Jones Industrial Average, USA), containing 28 assets and 1353 price observations (period: 07/05/1990 - 04/04/2016);
\item[2.] EUROSTOXX50 (Europe's leading blue-chip index, EU), containing 49 assets and 729 price observations (period: 22/04/2002-04/04/2016);
\item[3.] FTSE100 (Financial Times Stock Exchange, UK), containing 83 assets and 625 price observations (period: 19/04/2004-04/04/2016);
\item[4.] SP500 (Standard $\&$ Poor's, USA), containing 442 assets and 573 observations (period: 18/04/2005-04/04/2016).
\end{itemize}

Each dataset consists of weekly prices data. To evaluate the performance of \Mr{the} models in practice, we divide the observations in two sets, where the first one is regarded as the past (in-sample window), and so it is considered known, and the rest is regarded as the future (out-of-sample window), supposed unknown at the time of portfolio selection. The in-sample window is used for selecting the portfolio, while the out-of-sample one is used for testing the performance of the selected portfolio. In particular, in our experiments we use a \emph{rolling time window} scheme allowing for the possibility of rebalancing the portfolio composition during the holding period, at fixed intervals. Following \cite{JegTit_2001,ManOgrySpe_2007,PuerMadSco2020}, for each dataset we adopt a period of 52 weeks (one year) as in-sample window and we consider 12 weeks (three months) as out-of-sample, with rebalancing allowed every 12 weeks.

\Mr{For each dataset, in} all the portfolio selection models considered we set $\mu_0$ equal to the \Mr{average of the market portfolio return} in the in-sample \Mr{period}. When considering the RMT filtering method we choose the 5 largest eigenvalues ($p=5$), since, as reported in \cite{Schafer_etal_2010}, this choice yielded the best results in their experimental framework. For the Power Mapping method  we set $q = 1.25$ since this value in \cite{Schafer_etal_2010} provided the best results in the case without short-selling. \Mr{In the case of our \mrd{Scenario Filtering approach} we} decided to remove up to 5 scenarios \Mr{of each in-sample period, in order to} not increase the computational burden of our method and, at the same time, to not distort the dataset too much. Since each in-sample period contains 52 observations, the above decision corresponds to remove up to the $10\%$ (approximately) of the observations, which seems reasonable if one takes into account the frequency of occurrence of extreme events in the fat tail distributions that characterize returns distributions in financial markets. From the portfolio value performance viewpoint, we show that, in fact, this choice of $K$ is enough to outperform the other competing filtering models.

In the following tables we consider some classic out-of-sample performance measures described below:

\begin{itemize}
\item[1.] \emph{Average return} (\Mr{AvReturn}): it is defined as the average $E[\mu^{\rm{out}}(x)]$ of the out-of-sample returns of a portfolio \Mr{$x$. The larger is the value of the index, the better is the corresponding portfolio performance.}

\item[2.] \emph{Out-of-sample Variance} (V-Out): \Mr{it is the variance $\sigma^2(\mu^{\rm{out}}(x))$ of the out-of-sample return of a portfolio $x$.} The smaller is the value of the index, the better is the corresponding portfolio performance.

\item [3.] \emph{Sharpe Ratio} (\Mr{Sharpe}) (\cite{Sharpe1966,Sharpe1994}): it is defined as the ratio between the average of the out-of-sample return of a portfolio $x$, $\mu^{\rm{out}}(x)$, minus a constant risk free rate of return $r_f$ (that we set equal to 0), and its standard deviation, namely:
$$\frac{E[\mu^{\rm{out}}(x) - r_f]}{\sigma(\mu^{\rm{out}}(x))}.$$

\noindent In the bi-criteria optimization approach of the classical MVO model, the larger is the value of the index, the better is the portfolio performance.
\end{itemize}

\noindent \Mr{We note that, since we are adopting a rolling time window scheme, the portfolio $x$ in the above performance measures is not `static' as it is rebalanced in each in-sample period. In addition, we include in the tables the following information:}

\begin{itemize}
\item[4.] \Mr{\emph{Mean number of assets} (MeanAssets): it is the mean of the numbers of assets selected in portfolio $x$ in each in-sample period.
We consider that an asset $j$ is selected in portfolio $x$ if $x_j\geq 0.01$, $j=1,\ldots,n$.}

\item[5.] \emph{Mean time} (MeanTime): it is the mean of the CPU times for solving the considered model in each in-sample period. In the case of the heuristic procedure, it is the mean of the running times. In each in-sample period we set a time limit of 7200 seconds for solving the model. This time limit is reached by our MIQP model only in some instances of the SP500 dataset. In these cases in which the solver may not have been able to find an optimal solution within the time limit, we also provide a measure of the gap obatined.

\item [6.] \emph{Mean gap} (MeanGap): it is the mean of the relative gaps, in percentage, in each in-sample period. As commented above, it only applies to our MIQP model in some instances of the SP500 dataset.

\item [7.] \emph{Mean Relative Error} (MRE): it is the mean of the relative errors, in percentage, in each in-sample period. It only applies to our heuristic procedure as a measure of the relative difference between the optimal
solution obtained by solving \mrd{our MIQP model} and the best heuristic solution found. Small MRE values indicate that the solutions provided by the heuristic are close to the corresponding
optimal ones. For the SP500 dataset, MRE uses the best (possibly optimal) solutions found by \mrd{our MIQP model} within the time limit.
\end{itemize}

\noindent The models have been implemented in \Mr{MATLAB R2018a} and they make calls to XPRESS solver version 8.5 for solving the MIP problems. All experiments were run in a computer DellT5500 with a processor Intel(R) Xeon(R) with a CPU X5690 at 3.75 GHz and 48 GB of RAM memory.
\vskip 8 pt

\noindent \Mr{We start by showing the potential of our approach for lowering the out-of-sample risk (i.e., V-Out). We will only use for this demostration \mrd{our MIQP model}, as it corresponds to the exact implementation of our \mrd{Scenario Filtering approach} and given that it is enough for our illustrative purpose. Although similar results are obtained for \mrd{our heuristic algorithm}, a more detailed analysis of this method will be provided in the sequel.} We show in Table \ref{Table_0} the portfolio out-of-sample risk regardless of the effect of the portfolio return. This allows to better evaluate only the realized risk of optimized portfolios. To do this, in all models we \Mr{require} that the portfolio expected return be exactly equal to $\mu_0$ (see, e.g., \cite{Kondor_etal_2007}). The following table shows the \Mr{portfolio out-of-sample risk values} for all the \Mr{portfolio selection models and datasets considered}. Best values are reported in bold.

\begin{table}[htb]
\caption{\Mr{Out-of-sample risk} ($\cdot 10^{-4}$) with a \mr{12 weeks} rebalancing \Mr{when the expected return of the portfolio is set to $\mu_0$}.}
\label{Table_0}
\centering
\small
\begin{tabular}{|c||c|c|c|c|}
\hline
   & DJIA & EUROSTOXX  & FTSE100 & SP500 \\
\hline
\hline
Markowitz	    &	4.518	&	6.465	&	4.629	& 4.188	\\
\hline
\hline
\hline
RMT 	        &	4.583	&	6.389	&	4.631	& 4.215	\\
\hline
Power Mapping	&	4.375	&	6.398	&	4.398	& 4.052	\\
\hline
\hline
\mrd{Scenario Filtering}\\
\hline
$K=1$ 	        &	4.416	&	 \textbf{5.424}	&   \textbf{4.280}	& \textbf{3.976}	\\
$K=2$ 	        &	\textbf{4.307}	&	 5.627	&   4.491	& 4.233	\\
$K=3$ 	        &	4.408	&	 5.648	&   4.311	& 4.462	\\
$K=4$ 	        &	4.328	&	 5.909	&   4.342	& 4.521	\\
$K=5$ 	        &	4.352	&    5.803  &	4.397   & 4.554	\\
\hline
\end{tabular}
\end{table}

\vskip 8 pt
\noindent From the above table there seems to be a slight preference for our model. In fact, on all datasets, there always exists a value of $K$ by which we are able to produce portfolios with lower out-of-sample risk than the others. We observe that our \mrd{MIQP model} always provides the best realized risk values (in bold). We also remark that for three of the four datasets this result was obtained with $K=1$. The choice $K=1$ corresponds removing about only the $2\%$ of the observations in each in-sample period. This shows the impact of extreme observations in the data, and, therefore, it certifies the correctness of our approach.

\noindent On the other hand, the classical Markowitz model, from which all the approaches presented in this paper originate, is a bi-criteria optimization model, and, therefore, the comparison must be performed on the basis of both values of realized risk and return. To this aim, Tables \ref{Table_1}-\ref{Table_4} report the complete out-of-sample analysis based on all the performance measures introduced at the beginning of this section. In this case, in all the models the portfolio expected return constraint is modeled as an inequality, so that, it is possible to jointly assess the effect of return and risk in the selected portfolios. In this case, for the sake of completeness, in the following tables and for each dataset we also include, as a benchmark, the values of the performance measures of the market portfolio, referred as ``Market''.

\begin{table}[htb]
\caption{Out-of-sample performances for DJIA ($n=28$) with a \mr{12 weeks} rebalancing.}
\label{Table_1}
\centering
\small
\begin{tabular}{|c||c|c|c|c|c|}
\hline
   & AvReturn ($\cdot 10^{-3}$) & V-Out ($\cdot 10^{-4}$)  & Sharpe ($\cdot 10^{-2}$) &  MeanAssets & MeanTime (sec.)\\
\hline
\hline
Market     &   1.657   &   5.460   &  7.090  &    28  &  -        \\
\hline
\hline
Markowitz	&	1.908	&	4.036	&	9.500  &	9.9	& 0.031		\\
\hline
\hline
RMT 	&	1.842	&	4.039	&	9.164  &	9.7	&	0.038		\\
\hline
Power Mapping	&	1.930	&	\textbf{3.893}	&	9.780	&	11.9	 &	0.035	\\
\hline
\hline
\mrd{Scenario Filtering}\\
\hline
$K=1$	        &	1.915	&	   4.136	&        9.418 &			9.8	&	1.969		\\
$K=2$	        &	1.907	&	   4.096	&        9.420 &			9.7	&	2.806		\\
$K=3$            &	1.888	&	   4.099	&        9.325 &			9.8	&	7.033		\\
$K=4$	        &	1.924	&	   4.166	&        9.428 &			9.8	&	22.402		\\
$K=5$	        &	\textbf{2.032}	&  4.245	&	\textbf{\emph{9.864}}	 &	9.8 &	63.690		\\
\hline
\end{tabular}
\end{table}

\begin{table}[htb]
\caption{Out-of-sample performances for EUROSTOXX ($n=49$) with a \mr{12 weeks} rebalancing.}
\label{Table_2}
\centering
\small
\begin{tabular}{|c||c|c|c|c|c|}
\hline
   & AvReturn ($\cdot 10^{-3}$) & V-Out ($\cdot 10^{-4}$)  & Sharpe ($\cdot 10^{-2}$) &  MeanAssets & MeanTime (sec.)\\
\hline
\hline
Market	    &	0.799	&	8.732	& 2.703	 &	49 &	-		\\
\hline
\hline
Markowitz	&	1.778	&	4.902	&	\textbf{8.029} &	8.9	&	0.036		\\
\hline
\hline
RMT    &	1.666	&	4.899	&	7.527	&		9.1 &	0.041		\\
\hline
Power Mapping	&	1.705	&	\textbf{4.785}	&	7.793	&	11.1	 &	0.038	\\
\hline
\hline
\mrd{Scenario Filtering}\\
\hline
$K=1$ 	        &	1.724	&	5.250	&	7.526 &		9.6	&	3.253		\\
$K=2$ 	        &	1.810	&	5.388	&	7.798 &		9.7	&	4.502		\\
$K=3$ 	        &	1.790	&	5.500	&	7.633 &		9.9	&	7.951		\\
$K=4$ 	        &	1.840	&	5.546	&	7.812 &	    9.6	&	21.540		\\
$K=5$	  &	\textbf{1.887}	&	5.606	&	\emph{7.971}	&		9.9 &	61.881		\\
\hline
\end{tabular}
\end{table}

\begin{table}[htb]
\caption{Out-of-sample performances for FTSE100 ($n=83$) with a \mr{12 weeks} rebalancing.}
\label{Table_3}
\centering
\small
\begin{tabular}{|c||c|c|c|c|c|}
\hline
   & AvReturn ($\cdot 10^{-3}$) & V-Out ($\cdot 10^{-4}$)  & Sharpe ($\cdot 10^{-2}$) &  MeanAssets & MeanTime (sec.)\\
\hline
\hline
Market	    &	0.762	&	6.656	&	2.954  &	83 & -	   	    \\
\hline
\hline
Markowitz	&	1.852	&	4.339	&	8.892 &	12.7	& 0.042		\\
\hline
\hline
RMT 	&	1.577	&	4.333	&	7.578  &	14.4	 & 0.046	\\
\hline
Power Mapping	&	1.749	&	4.091	&	8.646 &	16.4 	&  0.047		\\
\hline
\hline
\mrd{Scenario Filtering}\\
\hline
$K=1$ 	        &	\textbf{2.406}	&	4.015	&	\textbf{\emph{12.001}} &	12.7	&	8.365		\\
$K=2$            &	2.152	&	\textbf{3.944}	&	10.840 &	13.2	& 13.877	\\
$K=3$            &	1.905	&	4.176	&	9.324 &	13.3	& 31.997		\\
$K=4$ 	        &	1.873	&	4.134	&	9.212 &	13.1	&	120.866		\\
$K=5$ 	        &	1.837	&	4.096	&	9.075 &	13.4	&	367.511		\\
\hline
\end{tabular}
\end{table}

\begin{table}[htb]
\caption{Out-of-sample performances for SP500 ($n=442$) with a \mr{12 weeks} rebalancing.}
\label{Table_4}
\centering
\small
\begin{tabular}{|c||c|c|c|c|c|c|}
\hline
   & AvReturn ($\cdot 10^{-3}$) & V-Out ($\cdot 10^{-4}$)  & Sharpe ($\cdot 10^{-2}$) &  MeanAssets & MeanTime (sec.) & MeanGap ($\%$)  \\
\hline
\hline
Market	    &	1.292	&	7.544	&	4.704     &	  442 &	-		    &	 -	\\
\hline
\hline
Markowitz	    &	1.560	&	3.603	&	8.220 &	16.4	    &	0.230	&	 -	\\
\hline
\hline
RMT 	&	1.619	&	\textbf{3.347}	&	8.847	    &	21.2	 & 0.263	&		-\\
\hline
Power Mapping	&	1.506	&	3.398	&	8.170	    &	21.8	 & 69.195	&		-\\
\hline
\hline
\mrd{Scenario Filtering}\\
\hline
$K=1$	        &	1.273	&	3.706	&	6.613	    &	17.6	 & 131.609	&	 0\\
$K=2$	        &	1.161	&	3.766	&	5.981	    &	17.4	 & 319.337	&	0\\
$K=3$	        &	1.320	&	3.924	&	6.662	&	17.2	 & 1055.618	&	0\\
$K=4$	        &	1.455	&	4.165	&	7.132	&	17.9	 & 4106.864	&	 15.488\\
$K=5$	  &	\textbf{2.016}	&	4.302	&	\textbf{\emph{9.720}}	&	18.1	 & 6174.948	&	 58.857	\\
\hline
\end{tabular}
\end{table}

\vskip 8 pt

\noindent In Tables \ref{Table_1}-\ref{Table_4} best values are in bold. We observe that our \mrd{MIQP model} always provides portfolios having (on average) the best out-of-sample performance in terms of return. On the other hand in this case, except for the FTSE100 dataset, the realized risk of our portfolios is worse than the realized risk provided by the portfolios found by the two alternative filtering models. However, in a bi-criteria framework and from an investor viewpoint, the Sharpe ratio, which evaluates the compromise between the return of a portfolio and the risk that the investor is affording, is the most significant quantity. In our experiments, the Sharpe ratio of our portfolios always outperforms the ratios provided by the Power Mapping and RMT models for all the datasets (see values in \emph{italic}).
\vskip 8 pt
\noindent About the ``MeanAssets'' column, in Tables \ref{Table_1}-\ref{Table_4} we note that the average number of selected stocks provided by our approach is the same as the one provided by the classical Markowitz model. On the other hand, the RMT and Power Mapping approaches tend to select slightly more assets than our model especially for large datasets. Limiting the number of selected stocks is often a requirement that come from real-world practice where the administration of a portfolio made up of a large number of assets, possibly with very small holdings for some of them, is clearly not desirable because of transactions costs, minimum lot sizes, complexity of management, or policy of the asset management companies.

\vskip 10 pt
\noindent As observed, and as evident from the \Mr{above} tables, our MIQP model is hard to solve at optimality especially for large financial datasets.
In general, computational times grow w.r.t. to the number of assets and the parameter $K$. From a practitioner viewpoint,
there is the need of computing portfolios having a good out-of-sample performance without too much waste of time.
Hence, in the following Tables \ref{Table_5}-\ref{Table_8} we report the same experimental analysis by applying
the heuristic procedures introduced in Section 3.2.


\begin{table}[htb]
\caption{Out-of-sample performances for DJIA ($n=28$) with a \mr{12 weeks} rebalancing applying the \mrd{heuristic procedure}.}
\label{Table_5}
\centering
\small
\begin{tabular}{|c||c|c|c|c|c|c|}
\hline
   & AvReturn ($\cdot 10^{-3}$) & V-Out ($\cdot 10^{-4}$)  & Sharpe ($\cdot 10^{-2}$) &  MeanAssets & MeanTime (sec.) &  MRE ($\%$) \\
\hline
\hline
Market       &   1.657   &   5.460   &  7.090    &  28 & -     &     -   \\
\hline
\hline
Markowitz	&	1.908	&	4.036	&	9.500	& 9.9	& 0.031	&	 -\\
\hline
\hline
RMT 	&	1.842	&	4.039	&	9.164	& 9.7 &	0.038	&	 -	\\
\hline
Power Mapping	&	1.930	&	\textbf{3.893}	&	9.780	& 11.9 & 	0.035	&	 -	\\
\hline
\hline
\mrd{Heuristic}\\
\hline
$K=1$	&	1.915	&	4.136   &	9.418	& 9.8	&	1.985	&		0	\\
$K=2$	&	1.897	&	4.103	&	9.367	& 9.7	&	4.206	&		0.073	\\
$K=3$	&	1.957	&	4.106	&	9.658	& 9.8	&	6.557	&		0.294	\\
$K=4$	&	1.980	&	4.176	&	9.690	& 9.8	&	8.870	&		0.453	\\
$K=5$	&	\textbf{2.069}	&	4.245	&	\textbf{\emph{10.043}}	& 9.8	&	11.155	&		0.575	\\
\hline
\end{tabular}
\end{table}

\begin{table}[htb]
\caption{Out-of-sample performances for EUROSTOXX ($n=49$) with a \mr{12 weeks} rebalancing applying the \mrd{heuristic procedure}.}
\label{Table_6}
\centering
\small
\begin{tabular}{|c||c|c|c|c|c|c|}
\hline
   & AvReturn ($\cdot 10^{-3}$) & V-Out ($\cdot 10^{-4}$)  & Sharpe ($\cdot 10^{-2}$) &  MeanAssets & MeanTime (sec.) &  MRE ($\%$) \\
\hline
\hline
Market	        &	0.799	&	8.732	& 2.703	 & 49	&	-	&	 -\\
\hline
\hline
Markowitz	&	1.778	&	4.902	&	\textbf{8.029}	&	 8.9	 & 0.036	&	 -\\
\hline
\hline
RMT             &	1.666	&	4.899	&	7.527	&	9.1 & 0.041	&		 -	\\
\hline
Power Mapping	&	1.705	&	\textbf{4.785}	&	7.793	& 11.1&	0.038	&	 -	\\
\hline
\hline
\mrd{Heuristic}\\
\hline
$K=1$	&	1.724	         &	5.250	&	7.526	&	 9.6	& 3.259	&		0	\\
$K=2$	&	1.810	         &	5.388	&	7.798	&	9.7	& 7.303	&		0	\\
$K=3$	&	1.803	         &	5.502	&	7.686	&	9.9	& 11.711	&		0.011	\\
$K=4$	&	\textbf{1.862}	         &	5.541	&	\emph{7.909}	& 9.8	&	16.110	&		0.160	\\
$K=5$	&	1.840	 &	5.620	&	7.763	& 9.9	&	20.493	&		0.247	\\
\hline
\end{tabular}
\end{table}

\begin{table}[htb]
\caption{Out-of-sample performances for FTSE100 ($n=83$) with a \mr{12 weeks} rebalancing applying the \mrd{heuristic procedure}.}
\label{Table_7}
\centering
\small
\begin{tabular}{|c||c|c|c|c|c|c|}
\hline
   & AvReturn ($\cdot 10^{-3}$) & V-Out ($\cdot 10^{-4}$)  & Sharpe ($\cdot 10^{-2}$) &  MeanAssets & MeanTime (sec.) &  MRE ($\%$) \\
\hline
\hline
Market	        &	0.762	&	6.656	&	2.954   & 83	& -	   &	 -    \\
\hline
\hline
Markowitz	&	1.852	&	4.339	&	8.892	& 12.7 &  0.042	&	-\\
\hline
\hline
RMT 	        &	1.577	&	4.333	&	7.578   & 14.4 & 0.046	&	 -	\\
\hline
Power Mapping	&	1.749	&	4.091	&	8.646	&  16.4 & 0.047	&	-	\\
\hline
\hline
\mrd{Heuristic}\\
\hline
$K=1$	&	\textbf{2.406}	&	4.015	&	\textbf{\emph{12.008}}	& 12.7	&	7.939	&		0	\\
$K=2$	&	2.149	&	\textbf{3.951}	&	10.811	&	13.3 	& 15.569	&		0.074	\\
$K=3$	&	1.876	&	4.176	&	9.179	& 13.2	&	22.999	&		0.206	\\
$K=4$	&	1.899	&	4.147	&	9.325	& 13.0	&	30.215	&		0.465	\\
$K=5$	&	1.992	&	4.107	&	9.830	& 13.2	&	37.339	&		0.498	\\
\hline
\end{tabular}
\end{table}

\newpage

\begin{table}[htb]
\caption{Out-of-sample performances for SP500 ($n=442$) with a \mr{12 weeks} rebalancing applying the \mrd{heuristic procedure}.}
\label{Table_8}
\centering
\small
\begin{tabular}{|c||c|c|c|c|c|c|}
\hline
   & AvReturn ($\cdot 10^{-3}$) & V-Out ($\cdot 10^{-4}$)  & Sharpe ($\cdot 10^{-2}$) &  MeanAssets & MeanTime (sec.) &  MRE ($\%$) \\
\hline
\hline
Market	    &	1.292	&	7.544	&	4.704      &	 442	    & -	&	 	 -	\\
\hline
\hline
Markowitz	    &	1.560	&	3.603	&	8.220	    &	16.4	 & 0.230	&		-\\
\hline
\hline
RMT 	&	1.619	&	\textbf{3.347}	&	8.847	    &	21.2	 & 0.263	&		-\\
\hline
Power Mapping	&	1.506	&	3.398	&	8.170	    &	21.8	 & 69.195	&		-\\
\hline
\hline
\mrd{Heuristic}\\
\hline
$K=1$	&	1.273	&	3.707	&	6.611	& 17.6	&	9.018	&		0	\\
$K=2$	&	1.202	&	3.828	&	6.144	& 17.4	&	17.641	&		0.113\\
$K=3$	&	1.364	&	4.008	&	6.814	& 17.3	&	26.040	&		0.097	\\
$K=4$	&	1.472	&	4.250	&	7.142	& 17.8	&	34.249	&		0.518	\\
$K=5$	&	\textbf{1.860}	&	4.265	&	\textbf{\emph{9.006}}	& 18.1	&	42.331	&		-0.743\\
\hline
\end{tabular}
\end{table}

\noindent We can see in Tables \ref{Table_5}-\ref{Table_8} that the performance results of our heuristic are in line with the ones obtained for the exact model. As expected, the computational times have been dramatically reduced while the relative error is on average no larger than the $1\%$. This shows the effectiveness of the heuristic procedure. We also note that in Table \ref{Table_8} for $K=5$ the heuristic was able to find solutions with a better value than the values provided by the exact model \Mr{(see the negative MRE)}, pointing out that the lager the value of parameter $K$, the harder is to find an optimal solution with the exact approach. Indeed, the computational times needed to solve the exact model grow exponentially with $K$, while this growth is linear in the case of the heuristic. Regarding the two versions of the heuristic described in Section 3.2, the CPU times reported in the tables have been obtained by implementing the Version 1 for DJIA and EUROSTOXX50 datasets, and the Version 2 for FTSE100 and SP500. Version 2 is therefore more suitable for large datasets in our setting. The CPU time needed to run the heuristic seems reasonable (less than one minute on average), and indeed it is lower than the one required by the Power Mapping filtering technique in the case of Table \ref{Table_8}. We conjecture that this increase in the computational time of the Power Mapping method is due to the fact that the correlation matrix after the transformation may be indefinite. To conclude this analysis of our heuristic, we note that the performance measures of the heuristic in Tables \ref{Table_5}-\ref{Table_8} sometimes improved the corresponding results of the exact method in Tables \ref{Table_1}-\ref{Table_4}. The possible capacity of the heuristic of avoiding \emph{overfitting} effects explains this fact.

\noindent Finally, to emphasize the out-of-sample performance of our approach, in the following figures we show the weekly out-of-sample portfolio values. For each dataset, $K$ is the number of filtered observations corresponding to the values in bold in column ``AvReturn'' in Tables \ref{Table_1}-\ref{Table_4}. \mrd{The red and dark blue lines report the weekly values of the portfolios obtained with our MIQP model (``Scenario Filtering'') and our heuristic algorithm, respectively.} Note that in Fig. 1 (c) the red and dark blue lines coincide since, as pointed out in Section 3.2, the solutions of the exact and heuristic methods coincide for $K=1$. From our computational experiments, one can conclude that our exact and heuristic algorithms outperform the alternative filtering methods on the analyzed real-world datasets.

\newpage

\begin{figure}[ht]
  \subfloat[DJIA $K=5$]{
	\begin{minipage}[c][1\width]{
	   0.49\textwidth}
	   \centering
	   \includegraphics[width=1.15\textwidth]{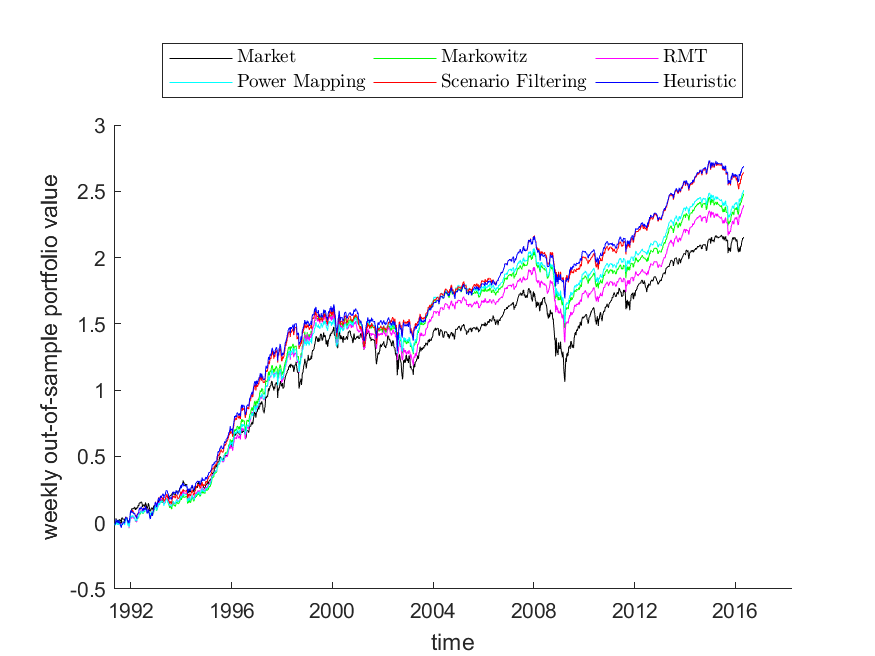}
	\end{minipage}}
 \hfill 	
  \subfloat[EUROSTOXX $K=5$]{
	\begin{minipage}[c][1\width]{
	   0.49\textwidth}
	   \centering
	   \includegraphics[width=1.15\textwidth]{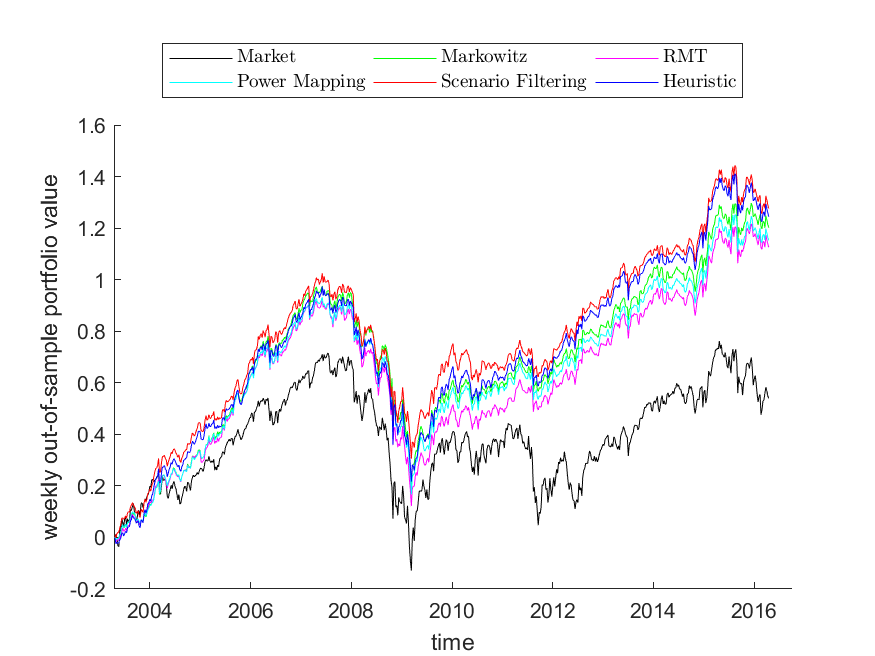}
	\end{minipage}}
\hfill
  \subfloat[FTSE100 $K=1$]{
	\begin{minipage}[c][1\width]{
	   0.49\textwidth}
	   \centering
	   \includegraphics[width=1.15\textwidth]{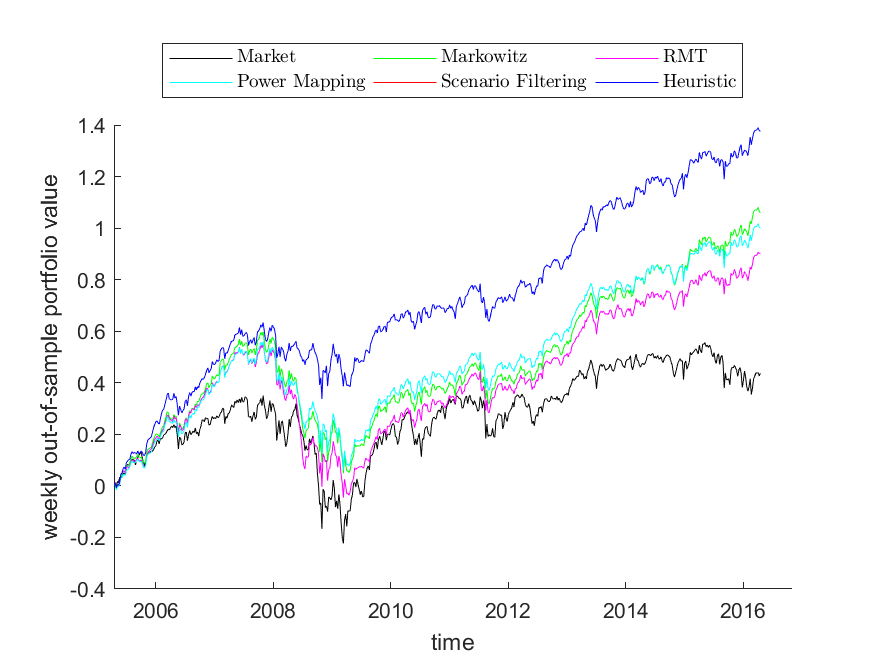}
	\end{minipage}}
 \hfill	
  \subfloat[SP500 $K=5$]{
	\begin{minipage}[c][1\width]{
	   0.49\textwidth}
	   \centering
	   \includegraphics[width=1.15\textwidth]{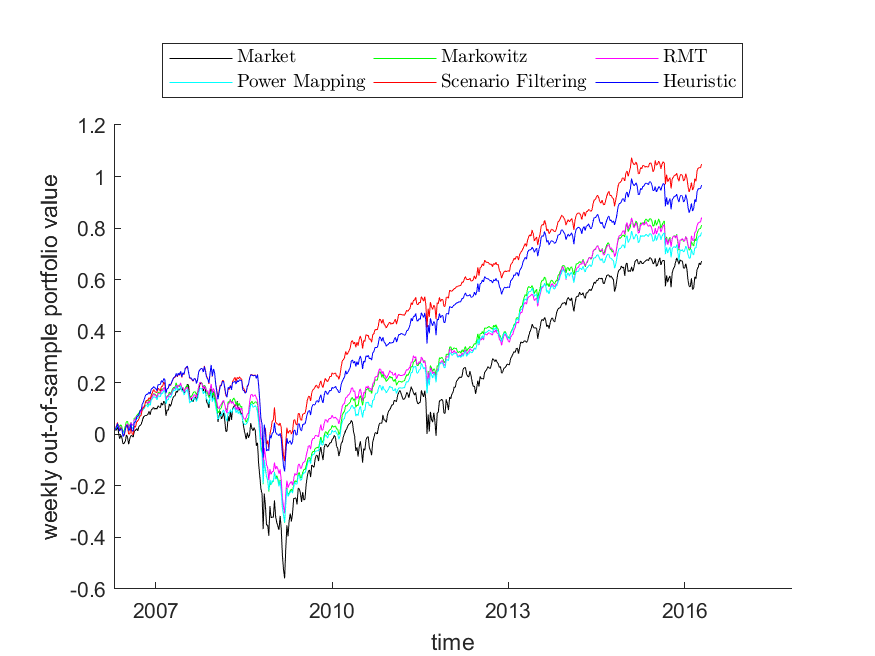}
	\end{minipage}}
\caption{Weekly out-of-sample portfolio values for the four datasets.}\label{Figure_OUT}
\end{figure}

\section{Conclusions}
Due to the finiteness of sample size of financial datasets, covariance matrices typically incorporate a huge amount of noise. Therefore, their use in portfolio selection maybe misleading. In this paper we consider some popular filtering procedures provided in the literature and compare them with a new approach based on combinatorial optimization. Our method is able to eliminate observations (outliers) in order to lower the in-sample variance and obtaining a good out-of-sample performance of the portfolios. We present a MIQP model and apply it to some real-world financial datasets. We show that our new combinatorial optimization approach to filtering is effective in hitting the goal of eliminating noise in the observed data. From a computational viewpoint, the two possible (exact and heuristic) strategies are able to find optimal or near optimal solutions in reasonable times even for large size financial datasets.

\begin{acknowledgements}
\JP{This research has been partially supported by Spanish Ministry of Education and Science/FEDER grant number  MTM2016-74983-C02-(01-02), and projects FEDER-US-1256951, CEI-3-FQM331 and  \textit{NetmeetData}: Ayudas Fundaci\'on BBVA a equipos de investigaci\'on cient\'ifica 2019.}
\end{acknowledgements}

\bibliographystyle{elsarticle-harv}

\end{document}